\documentclass[12pt]{amsart}
\usepackage[pagewise]{lineno}\linenumbers
\usepackage{tikz}
\usepackage{amssymb}
\usepackage{amsmath}
\usepackage{mathrsfs}
\usepackage{mathbbol}
\usepackage{amsfonts}
\usepackage{amssymb,amsmath}
\usepackage{graphicx}
\graphicspath{{./images/}}
\def\R {\mathbb{R}}

\def\d{{\,\rm d}}

\newtheorem{proposition}{Proposition}[section]
\newtheorem{theorem}[proposition]{Theorem}

\newtheorem{lemma}[proposition]{Lemma}
\theoremstyle{definition}

\newtheorem{remark}[proposition]{Remark}
\numberwithin{equation}{section}

\def \no#1#2#3 {{\bf #1} (#3), #2.}
\def \eds#1#2#3 {#1, #2, #3.}

\title[ Analytic radius for BBM equation]
{ \bf
Improved lower bounds of analytic radius for the Benjamin-Bona-Mahony equation}

\author[M. Wang]
{Ming Wang}

\address{Ming Wang
\newline\indent
School of Mathematics and Physics, China University of Geosciences
\newline\indent
Wuhan, 430074, P.R. China and
\newline\indent Center for Mathematical Sciences, China University of Geosciences
\newline\indent Wuhan, 430074, P.R. China
}
\email{mwang@cug.edu.cn}



\begin{document}

\begin{abstract}
This paper is devoted to the spatial analyticity of the solution of the BBM equation on the real line with an analytic initial data. It is shown that the analytic radius has a lower bound like $t^{-\frac{2}{3}}$ as time $t$ goes to infinity, which is an improvement of previous results. The main new ingredient is a higher order almost conservation law in analytic spaces. This is proved by introducing an equivalent analytic norm with smooth symbol and establishing some algebra identities of higher order polynomials.
\end{abstract}

\maketitle


\section{Introduction and main results}

In this paper, we are interested in the long time behavior of the spatial analytic radius for the Benjamin-Bona-Mahony (BBM) equation on the real line $\R$
\begin{align}\label{bbm}
u_t-u_{txx} +  u_x + uu_x=0, \quad u(0,x)=u_0(x).
\end{align}
Here $u(t,x)$ is a real-valued function of $(t,x)\in \R^2$, and we adopt the notation $u_{txx}=\partial_t\partial_x^2u$. The equation \eqref{bbm} was introduced in \cite{bona72} to study the dynamics of small-amplitude surface water waves propagating unidirectionally.  The BBM equation \eqref{bbm} admits the $H^1(\R)$ conservation law, namely
\begin{align}\label{equ-intro-law}
\int  (u^2+u_x^2)(t)\d x = \int  (u_0^2+u_{0x}^2)\d x, \quad \forall t\in \R.
\end{align}
Based on the identity \eqref{equ-intro-law}, it is shown in \cite{bona09,pan11,wangm16,Dai,Ban} that the BBM equation \eqref{bbm} is globally well posed in $H^s(\R)$ for all $s\geq 0$, and ill poseded in $H^s(\R)$ for $s<0$.

Recently, there is a growing interest in the well posedness of dispersive equations in analytic (or Gevrey) spaces $G^\sigma(\R)$. Here, for every $\sigma>0$, $G^\sigma(\R)$ denotes the Banach space endowed with the norm
$$
\|u\|_{G^\sigma}=\left(\int_\R e^{2\sigma|\xi|}|\widehat{u}(\xi)|^2\d \xi \right)^{\frac{1}{2}}.
$$
According to the Paley-Wiener theorem \cite{Ka}, every function belongs to $G^\sigma(\R)$ can be extend an analytic function in the strip $\{z\in \mathbb{C}: |\mbox{Im } z|<\sigma\}$. The local well posedness in $G^\sigma$ can be proved similarly as that in Sobolev space $H^s(\R)$. But the global well posedness in $G^\sigma$ is far from clear. The main difficulty lies in that there is no conservation law as \eqref{equ-intro-law} in analytic space $G^\sigma$.

There are a few works devoted to the following more accessible question
$$
\mbox{ If } u_0\in G^{\sigma_0}, \mbox{ then for what kind of } \sigma(t) \mbox{ such that } u(t)\in G^{\sigma(t)}?
$$
In \cite{BG}, Bona and Gruji\'{c} proved that the analytic radius of solutions to \eqref{bbm} satisfies
\begin{align}\label{equ-lower-1}
\sigma(t)\geq c\,t^{-1}, \quad t\to \infty.
\end{align}
In the proof, it is shown that the solution satisfies some energy inequalities in terms of $G^{\sigma(t)}$ type norms, thus the solution $u(t)$ belongs to $G^{\sigma(t)}$ provided that $\sigma(t)$ enjoys an ordinary equation (ODE), solving the ODE gives the lower bound \eqref{equ-lower-1}. Very recently, Himonas and Petronilho \cite{HP} established the same analytic radius bound as \eqref{equ-lower-1} to the BBM equation on torus. The strategy in  \cite{HP} is as follows:
\begin{itemize}
  \item[(1)] Prove a local well-posedness in $G^{\sigma,1}$ (see \eqref{equ-ana-norm-1} for definition) with a lifespan $\delta>0$;
  \item[(2)] Establish an almost conservation law in $G^{\sigma,1}$, namely\footnote{In fact, letting $\sigma$ go to $0$, the inequality \eqref{equ-intro-1} reduces to the conservation law \eqref{equ-intro-law}.}
\begin{align}\label{equ-intro-1}
\|u(\delta)\|^2_{G^{\sigma,1}}\leq \|u_0\|^2_{G^{\sigma,1}}+ C\sigma\|u_0\|^2_{G^{\sigma,1}};
\end{align}
  \item[(3)] Use repeatedly the local well-posedness and the almost conservation law on the intervals $[0, \delta], [\delta, 2\delta], \cdots$, and show that the solution satisfies \eqref{equ-lower-1}.
\end{itemize}
This approach is introduced by Selberg and Tesfahun in \cite{ST15}, can be understood a variant $I$-method \cite{Tao} in analytic spaces. The method is powerful and has been used to establish analytic radius lower bounds for KdV equations \cite{Sel17,Tes17,Tes19,wangm19,wangm22}, KdV-BBM equations \cite{pan20,Tes22-N} and  other dispersive equations \cite{Seo20,Seo21,H21,HP21,shan,liu-wang22}. For more results on the analytic radius, we refer to the survey \cite{HP20}.

The main goal in this paper is give a refinement of the bound \eqref{equ-lower-1}.

\begin{theorem}\label{thm-1}
Assume that $u_0\in G^{\sigma_0}$ with some $\sigma_0>0$. Then the BBM equation \eqref{bbm} has a global solution $u\in C([0,\infty);G^{\sigma(t)})$ with
$$
\sigma(t)\geq c\,t^{-\frac{2}{3}}, \qquad t\to \infty,
$$
where $c>0$ is a constant depending only on $\sigma_0$ and $\|u_0\|_{G^{\sigma_0}}$.
\end{theorem}

The proof of Theorem \ref{thm-1} still relies on the almost conservation law strategy. The main new ingredient is to establish a higher order conservation law (compared to \eqref{equ-intro-1})
\begin{align}\label{equ-intro-2}
\|u(\delta)\|^2_{G^{\sigma,1}}\leq \|u_0\|^2_{G^{\sigma,1}}+ C\sigma^{\frac{3}{2}}\|u_0\|^2_{G^{\sigma,1}}.
\end{align}
Roughly speaking, to obtain a better conservation law, one should exploit as possible as we can the cancelation property of the nonlinear term $uu_x$ in analytic space $G^\sigma$. To this end, we first introduce an equivalent norm
$$
\|u\|_{G^\sigma}\sim \|\cosh(\sigma \xi)\widehat{u}\|_{L^2(\R)},
$$
where $\cosh \xi$ is the hyperbolic cosine function, namely
\begin{align}\label{equ-intro-m}
\cosh(\sigma\xi)=\frac{1}{2}(e^{\sigma \xi}+e^{-\sigma \xi})=\sum_{k=0}\frac{\sigma^{2k}}{(2k)!}\xi^{2k}.
\end{align}
Then we reduce  the cancelation of the nonlinear term to that of the polynomial
$$
\xi_1^{2k+1}+\xi_2^{2k+1}+\xi_3^{2k+1} \quad \mbox{ on the hyperplane }\quad \xi_1+\xi_2+\xi_3=0.
$$
After a delicate factorization of the polynomial, we get the conservation law \eqref{equ-intro-2}. In this way, we obtain the improvement lower bound $\sigma(t)\geq ct^{-2/3}$.

Very recently, Dufera, Mebrate and Tesfahun \cite{Tes22,Tes22-b} presented a new idea to establish higher order conservation law. In fact, they introduced the equivalent norm (compared with \eqref{equ-intro-m})
$$
\|u\|_{G^\sigma} \sim \left\|\cosh (\sigma |\xi|)\widehat{u} \right\|_{L^2(\R)}
$$
and deduced the almost conservation law based  on the inequality
$$
\cosh (\sigma |\xi|)-1\leq (\sigma|\xi|)^{2\alpha} \cosh(\sigma|\xi|), \quad 0\leq \alpha\leq 1.
$$
But, applying the approach, it seems not easy to deduce the analytic radius lower bound $t^{-\frac{2}{3}}$ in Theorem \ref{thm-1}, see Remark \ref{rem-Tes}.

The paper is organized as follows. We prove Theorem \ref{thm-1} in Section 2 with aid of an almost conservation law, which is established in Section 3. In Section 4, we apply our method to the fractional BBM equation.

Finally, we say a few words about the notations. Throughout, for two nonnegative quantities $A,B$, we write $A\lesssim B$ if $A\leq CB$ for some unimportant constant $C>0$. If both $A\lesssim B$ and $B\lesssim A$ hold, we shall write $A\sim B$. The Fourier transform is given by
$$
\widehat{f}(\xi) = \int_\R e^{-ix\xi}f(x)\d x.
$$
For every $s\in \R$, the Sobolev space $H^s(\R)$ is the Banach space endowed with the norm
$$
\|f\|_{H^s}=\left(\int_\R (1+|\xi|)^{2s}|\widehat{f}(\xi)|^2\d \xi \right)^{\frac{1}{2}}.
$$

\section{Proof of Theorem \ref{thm-1}: Reduction to almost conservation law}

Let $\sigma>0$. Define a Fourier multiplier operator
\begin{align}\label{equ-law-1}
Iu = \mathcal {F}^{-1}\Big( m(\xi)\widehat{u}(\xi) \Big),
\end{align}
where $\mathcal {F}^{-1}$ denotes the inverse of Fourier transform, and the symbol
\begin{align}\label{equ-law-2}
m(\xi)= \frac{1}{2}(e^{\sigma \xi}+e^{-\sigma \xi}), \quad \xi \in \R.
\end{align}
Clearly, we have the bound
\begin{align}\label{equ-law-3}
\frac{1}{2}e^{\sigma|\xi|}\leq m(\xi)\leq e^{\sigma|\xi|}, \quad \xi \in \R.
\end{align}
It follows from \eqref{equ-law-3} that
\begin{align}\label{equ-loc-1}
\frac{1}{2}\|u\|_{G^\sigma}\leq \|Iu\|_{L^2(\R)}\leq \|u\|_{G^\sigma}.
\end{align}
So $\|Iu\|_{L^2(\R)}$ is an equivalent norm with $\|u\|_{G^\sigma}$, but $I$ is defined by a smooth Fourier symbol, this is helpful when we consider the almost conservation law in the next section.
%

Denote by
\begin{align}\label{equ-loc-2}
\varphi(D)=\partial_x(1-\partial_x^2)^{-1}.
\end{align}
Acting both sides of \eqref{bbm} gives
\begin{align}\label{bbm-I-1}
u_t+ \varphi(D)u+  \frac{1}{2}\varphi(D)(u^2)=0, \quad u(0,x)=u_0(x).
\end{align}
The integral version of \eqref{bbm-I-1} reads
\begin{align}\label{bbm-inte}
u(t)=e^{-t\varphi(D)}u_0-\frac{1}{2}\int_0^te^{-(t-\tau)\varphi(D)}\varphi(D)(u^2)\d \tau.
\end{align}
Here $e^{-t\varphi(D)}$ is the $C_0$ semigroup generated by $-\varphi(D)$, which can also be understood as a Fourier multiplier with symbol $e^{-it\xi(1+|\xi|^2)^{-1}}$. Clearly, since the modulus of $e^{-it\xi(1+|\xi|^2)^{-1}}$ is $1$,  we have
\begin{align}\label{bbm-loc-3}
 \|e^{-t\varphi(D)}u\|_{H^s(\R)}=\|u\|_{H^s(\R)}, \quad \forall t\in \R, s\in \R.
\end{align}
Moreover, we need the following bilinear estimate.
\begin{lemma}\label{lem-bi}
For all $\sigma>0$, we have
$$
\|\varphi(D)I(uv)\|_{H^1(\R)}\lesssim \|Iu\|_{H^1(\R)}\|Iv\|_{H^1(\R)}.
$$
\end{lemma}
\begin{proof}
We first note that $\varphi(D)$ is bounded from $L^2(\R)$ to $H^1(\R)$. In fact, by Plancherel theorem we have
$$
\|\varphi(D)u\|_{H^1(\R)}\lesssim  \|(1+|\xi|)i\xi(1+|\xi|^2)^{-1}\widehat{u}\|_{L_\xi^2(\R)}\lesssim \|\widehat{u}\|_{L_\xi^2(\R)}\lesssim \|u\|_{L^2(\R)}.
$$
It follows that
\begin{align}\label{equ-bi-1}
\|\varphi(D)I(uv)\|_{H^1(\R)}\lesssim \|I(uv)\|_{L^2(\R)}.
\end{align}
Since the Fourier transform of $I(u^2)$ is
$$
\int_{\xi=\xi_1+\xi_2}m(\xi)\widehat{u}(\xi_1)\widehat{v}(\xi_2)
$$
and the function $\xi\mapsto (1+|\xi|)^{-1}$ belongs to $L^2(\R)$, we deduce that
\begin{align}\label{equ-bi-2}
 \|I(uv)\|_{L^2(\R)}&\lesssim \|\int_{\xi=\xi_1+\xi_2}m(\xi)\widehat{u}(\xi_1)\widehat{v}(\xi_2)\|_{L_\xi^2(\R)}\nonumber\\
 &\lesssim \sup_{\xi\in \R} \;\int_{\xi=\xi_1+\xi_2}(1+|\xi|)m(\xi)|\widehat{u}(\xi_1)\widehat{v}(\xi_2)|.
\end{align}
By the definition \eqref{equ-law-2},  it is easy to check that
\begin{align}\label{equ-bi-3}
m(\xi)\leq 2m(\xi_1)m(\xi_2), \quad \mbox{ if } \xi=\xi_1+\xi_2.
\end{align}
Moreover, we have
\begin{align}\label{equ-bi-4}
1+|\xi|\leq (1+|\xi_1|)(1+|\xi_2|), \quad \mbox{ if } \xi=\xi_1+\xi_2.
\end{align}
Plugging \eqref{equ-bi-3}-\eqref{equ-bi-4} into \eqref{equ-bi-2} gives that
\begin{align}\label{equ-bi-5}
 \|I(uv)\|_{L^2(\R)} \lesssim  \sup_{\xi\in \R} \; \int_{\xi=\xi_1+\xi_2}(1+|\xi_1|)m(\xi_1)|\widehat{u}(\xi_1)|(1+|\xi_2|)m(\xi_2)|\widehat{v}(\xi_2)|.
\end{align}
Applying Cauchy-Schwarz inequality, we deduce from \eqref{equ-bi-5} that
\begin{align*}
\|I(uv)\|_{L^2(\R)}&\lesssim \|(1+|\xi_1|)m(\xi_1)|\widehat{u}(\xi_1)|\|_{L^2}\|(1+|\xi_2|)m(\xi_2)|\widehat{v}(\xi_2)|\|_{L^2}\\
&\lesssim \|Iu\|_{H^1(\R)}\|Iv\|_{H^1(\R)}.
\end{align*}
This gives the desired bound.
\end{proof}

Now we prove the local well posedness of the BBM equation \eqref{bbm} in analytic space.

\begin{proposition}\label{prop-loc}
Let $\sigma>0$ and $I$ be defined by \eqref{equ-law-1}. Then for every initial data $u_0$ satisfying $Iu_0\in H^1(\R)$, there exists a unique solution of the BBM equation \eqref{bbm} satisfying $Iu\in C([0,\delta];H^1(\R))$ and
\begin{align}\label{bbm-loc-4}
\sup_{t\in [0,\delta]}\|Iu(t)\|_{H^1(\R)}\leq 2\|Iu_0\|_{H^1(\R)},
\end{align}
and the lifespan $\delta$ satisfies
\begin{align}\label{bbm-loc-5}
\delta\sim (\|Iu_0\|_{H^1(\R)})^{-1}.
\end{align}
\end{proposition}
\begin{proof}
It suffices to prove the well posedness for the integral equation \eqref{bbm-inte}. We shall use the contraction mapping principle.  Consider the mapping
\begin{align}\label{bbm-loc-6}
\Gamma u = e^{-t\varphi(D)}u_0-\frac{1}{2}\int_0^te^{-(t-\tau)\varphi(D)}\varphi(D)(u^2)\d \tau
\end{align}
on the ball
\begin{align}\label{bbm-loc-7}
\mathcal {B}=\Big\{u:  \sup_{t\in[0,\delta]}\|Iu(t)\|_{H^1(\R)}\leq 2\|Iu_0\|_{H^1(\R)}  \Big\}.
\end{align}
On one hand, if $u\in \mathcal {B}$, then by \eqref{bbm-loc-3} and Lemma \ref{lem-bi}
\begin{align}\label{bbm-loc-8}
 \sup_{t\in [0,\delta]}\|I\Gamma u\|_{H^1(\R)}&\leq \|e^{-t\varphi(D)}Iu_0\|_{H^1(\R)}+\frac{1}{2}\int_0^\delta \|e^{-(t-\tau)\varphi(D)}\varphi(D)I(u^2)\|_{H^1(\R)}\d \tau\nonumber\\
 & \leq \|Iu_0\|_{H^1(\R)}+c\delta \sup_{t\in[0,\delta]} \|Iu(t)\|^2_{H^1(\R)}\nonumber\\
 &\leq \|Iu_0\|_{H^1(\R)}+4c\delta\|Iu_0\|^2_{H^1(\R)}.
\end{align}
Moreover, if $u,v\in \mathcal {B}$, then
\begin{align}\label{bbm-loc-9}
 \sup_{t\in [0,\delta]}\|I\Gamma u-I\Gamma v\|_{H^1(\R)}&\leq  \frac{1}{2}\int_0^\delta \|e^{-(t-\tau)\varphi(D)}\varphi(D)I(u^2-v^2)\|_{H^1(\R)}\d \tau\nonumber\\
 & \leq  c\delta \sup_{t\in[0,\delta]} \|Iu(t)+Iv(t)\|_{H^1(\R)}\|Iu(t)-Iv(t)\|_{H^1(\R)}\nonumber\\
 &\leq  4c\delta\|Iu_0\|_{H^1(\R)}\sup_{t\in[0,\delta]} \|Iu(t)-Iv(t)\|_{H^1(\R)}.
\end{align}
Thanks to \eqref{bbm-loc-8}-\eqref{bbm-loc-9}, if we choose
$$
\delta = \frac{1}{8c\|Iu_0\|_{H^1(\R)}},
$$
then $\Gamma:\mathcal {B}\mapsto \mathcal {B}$ is a contraction mapping. This proves the uniqueness and existence of solution satisfying $Iu\in L^\infty((0,\delta);H^1(\R))$. Applying $I$ on both sides of \eqref{bbm-I-1} we get
\begin{align}\label{bbm-loc-10}
(Iu)_t+ \varphi(D)(Iu)+  \frac{1}{2}\varphi(D)I(u^2)=0.
\end{align}
Using Lemma \ref{lem-bi} again, it follows from \eqref{bbm-loc-10} that $Iu_t\in L^\infty((0,\delta);H^1(\R))$. By Sobolev embedding, we infer that $Iu\in C([0,\delta)];H^1(\R))$. This completes the proof.
\end{proof}

Now we are going to study the spatial analytic radius for large time. To this end, we shall write from now on
$$
I_\sigma u = \mathcal {F}^{-1}\Big( m(\xi)\widehat{u}(\xi) \Big), \quad  m(\xi)= \frac{1}{2}(e^{\sigma \xi}+e^{-\sigma \xi}), \quad \xi \in \R
$$
to emphasized the role of $\sigma>0$.

\begin{proposition}[Based on almost conservation law in Proposition \ref{prop-law}]\label{prop-global}
Assume that $\sigma_0>0$ and $I_{\sigma_0}u_0\in H^1(\R)$. Then when $u$ solves \eqref{bbm},
$$
I_{\sigma(t)}u(t)\in H^1(\R) \mbox{ with } \sigma(t)\geq ct^{-\frac{2}{3}}, \qquad t\to \infty,
$$
where $c>0$ is a constant depending only on $\sigma_0$ and $\|I_{\sigma_0}u_0\|_{H^1(\R)}$.
\end{proposition}
\begin{proof}
Let  $\sigma_0>0$ and $I_{\sigma_0}u_0\in H^1(\R)$. According to Proposition \ref{prop-loc}, the BBM equation has a unique solution such that $I_{\sigma_0}u\in C([0,\delta];H^1(\R))$ and
\begin{align}\label{equ-gl-1}
\sup_{t\in[0,\delta]}\|I_{\sigma_0}u(t)\|_{H^1}\leq 2\|I_{\sigma_0}u_0\|_{H^1},
\end{align}
with
\begin{align}\label{equ-gl-2}
\delta = \frac{1}{8C_1\|I_{\sigma_0}u_0\|_{H^1}}.
\end{align}
Also, thanks to the almost conservation law in Proposition \ref{prop-law},
\begin{align}\label{equ-gl-3}
 \sup_{t\in[0,\delta]}\|I_{\sigma_0}u(t)\|^2_{H^1}\leq \|I_{\sigma_0}u_0\|^2_{H^1}+C_2\delta \sigma^{\frac{3}{2}}\|I_{\sigma_0}u_0\|^3_{H^1}.
\end{align}

Fix $T\gg 1$. We shall use Proposition \ref{prop-loc} and Proposition \ref{prop-law} repeatedly to extend the estimates \eqref{equ-gl-1} and \eqref{equ-gl-3} to the interval $[0,T]$. With $\delta$ given by \eqref{equ-gl-2}, there exists a unique integer $n$ such that $T\in[n\delta,(n+1)\delta)$. We claim that if $k\in \{1,2,\cdots,n\}$ and
\begin{align}\label{equ-gl-4}
 \sigma = \min\left\{\sigma_0,  \Big( \frac{2C_1}{C_2(n+1)} \Big)^{\frac{2}{3}}\right\},
\end{align}
then
\begin{align}\label{equ-gl-4}
\sup_{t\in[0,k\delta]}\|I_{\sigma }u(t)\|_{H^1}\leq 2\|I_{\sigma_0}u_0\|_{H^1},
\end{align}
and
\begin{align}\label{equ-gl-5}
\sup_{t\in[0,k\delta]}\|I_{\sigma }u(t)\|^2_{H^1}\leq \|I_{\sigma_0}u_0\|^2_{H^1}+k C_2\delta 2^3\sigma^{\frac{3}{2}}\|I_{\sigma_0}u_0\|^3_{H^1}.
\end{align}
We shall argue by the induction. Clearly, the inequalities \eqref{equ-gl-4} and \eqref{equ-gl-5} in case $k=1$ follows from \eqref{equ-gl-1} and \eqref{equ-gl-3}, respectively. Now we assume that \eqref{equ-gl-4} and \eqref{equ-gl-5} hold for some $k\in \{1,2,\cdots,n\}$, then we show that they also hold for $k+1$. In fact, applying  Proposition \ref{prop-loc} on time interval $[k\delta,(k+1)\delta]$,  we have
\begin{align}\label{equ-gl-6}
\sup_{t\in[k\delta,(k+1)\delta]}\|I_{\sigma}u(t)\|_{H^1}&\leq \|I_{\sigma}u(k\delta)\|^2_{H^1}+C_2\delta \sigma^{\frac{3}{2}}\|I_{\sigma}u(k\delta)\|^3_{H^1}\nonumber\\
\mbox{ by \eqref{equ-gl-4} }&\leq  \|I_{\sigma}u(k\delta)\|^2_{H^1}+C_2\delta \sigma^{\frac{3}{2}}(2\|I_{\sigma_0}u_0\|_{H^1})^3\nonumber\\
\mbox{ by \eqref{equ-gl-5} }&\leq  \|I_{\sigma_0}u_0\|^2_{H^1}+k C_2\delta 2^3\sigma^{\frac{3}{2}}\|I_{\sigma_0}u_0\|^3_{H^1}+C_2\delta \sigma^{\frac{3}{2}}(2\|I_{\sigma_0}u_0\|_{H^1})^3\nonumber\\
&= \|I_{\sigma_0}u_0\|^2_{H^1}+ (k+1) C_2\delta 2^3\sigma^{\frac{3}{2}}\|I_{\sigma_0}u_0\|^3_{H^1}.
\end{align}
This proves \eqref{equ-gl-5} for $k+1$. Moreover, by \eqref{equ-gl-4} and \eqref{equ-gl-6}, we have
$$
\sup_{t\in[0,(k+1)\delta]}\|I_{\sigma }u(t)\|_{H^1}\leq 2\|I_{\sigma_0}u_0\|_{H^1},
$$
which verifies \eqref{equ-gl-4} for $k+1$, provided that
\begin{align}\label{equ-gl-7}
 (k+1) C_2\delta 2^3\sigma^{\frac{3}{2}}\|I_{\sigma_0}u_0\|^3_{H^1}\leq \|I_{\sigma_0}u_0\|^2_{H^1}.
\end{align}
Recall that \eqref{equ-gl-2}, the inequality \eqref{equ-gl-7} is equivalent to
$$
\sigma\leq \Big( \frac{2C_1}{C_2(k+1)} \Big)^{\frac{2}{3}}.
$$
But this follows from our choice \eqref{equ-gl-4}.

So we conclude that $I_\sigma u(T)\in H^1(\R)$ with
$$
\sigma =\Big( \frac{2C_1}{C_2(n+1)} \Big)^{\frac{2}{3}}\geq \Big( \frac{2C_1}{C_2(\frac{T}{\delta}+1)} \Big)^{\frac{2}{3}}
$$
since $n\geq T/\delta$. Thus
$$
\sigma \geq cT^{-\frac{2}{3}}, \quad \mbox{ for $T$ large.}
$$
This completes the proof.
\end{proof}

\begin{proof}[\bf Proof of Theorem \ref{thm-1}]
Assume that $u_0\in G^{\sigma_0}$ for some $\sigma_0>0$. Then we have $I_{\frac{\sigma_0}{2}}u_0\in H^1(\R)$. In fact,
$$
\|I_{\frac{\sigma_0}{2}}u_0\|_{H^1(\R)}= \|(1+|\xi|)2^{-1}(e^{\frac{\sigma_0}{2}\xi}+e^{-\frac{\sigma_0}{2}\xi})\widehat{u_0}(\xi)\|_{L^2(\R)}\lesssim \|e^{\sigma_0|\xi|}\widehat{u_0}(\xi)\|_{L^2(\R)}\lesssim \|u_0\|_{G^{\sigma_0}}.
$$
Now according to Proposition \ref{prop-global}, we know $I_{\sigma(t)}u(t)\in H^1(\R)$ for all $t>0$ with
\begin{align}\label{equ-global-1}
\sigma(t)\geq ct^{-\frac{2}{3}}, \quad t\to \infty,
\end{align}
where $c>0$ depends only on $\sigma_0$ and $\|u_0\|_{G_{\sigma_0}}$. Recall that the norm $\|I_\sigma u\|_{L^2(\R)}$ is equivalent to $\|u\|_{G^\sigma}$, see \eqref{equ-loc-1}, we know
$$
u\in L^\infty((0,\infty);G^{\sigma(t)})
$$
with $\sigma(t)$ satisfying \eqref{equ-global-1}. Moreover, similar to Lemma \ref{lem-bi}, we can show that
$$
\|\partial_x(1-\partial_x^2)^{-1}(uv)\|_{G^\sigma}\lesssim \|u\|_{G^\sigma}\|v\|_{G^\sigma}.
$$
With this bilinear estimate in hand, proceeding as Proposition \ref{prop-loc}, one can show that, for every $t_0\geq 0$, the solution of BBM equation satisfies
$$
u\in C([t_0,t_0+\delta];G^{\sigma(t_0)})
$$
for some $\delta>0$. Thus, if we assume that $\sigma(t)$ is a non-increasing function satisfying \eqref{equ-global-1}, then $u\in C([0,\infty);G^{\sigma(t)})$. This completes the proof of Theorem \ref{thm-1}.
\end{proof}

\section{Almost conservation law}

In this section, we shall prove an almost conservation law of the BBM equation \ref{bbm} in analytic space $G^\sigma$. This plays a key role in the proof of Theorem \ref{thm-1}, see Section 2. As mentioned in the introduction, the main idea is to reduce the cancelation property of the nonlinear term in analytic space to that of higher order polynomials. This strategy has been used in our previous work \cite{wangm19} for the KdV equation.

Using the Taylor expansion
$$
e^{\sigma \xi} = \sum_{k=0}^\infty \frac{(\sigma \xi)^k}{k!}, \quad \xi \in \R,
$$
we deduce from the definition \eqref{equ-law-2} that
\begin{align}\label{equ-law-4}
m^2(\xi)=\frac{1}{4}(e^{2\sigma \xi}+e^{-2\sigma \xi}+2) =1+ \frac{1}{2} \sum_{k=1}^\infty \frac{(2\sigma \xi)^{2k}}{(2k)!}, \quad \xi \in \R.
\end{align}
Multiplying \eqref{bbm} with $2I^2u$ and integrating over $\R$, we obtain
\begin{align}\label{equ-law-5}
 \frac{\d}{\d t}\int_\R \Big( (Iu)^2+(\partial_xIu)^2 \Big)\d x=-2(uu_x,I^2u),
\end{align}
where $(\cdot,\cdot)$ denotes the inner product in $L^2(\R)$, and we used the fact
$$
(u_x,I^2u)=\int_\R \frac{1}{2}\partial_x(Iu)^2\d x=0.
$$
So the key point is to estimate $2(uu_x,I^2u)$. Using integration by parts and Paserval formula, we have
\begin{align}\label{equ-law-6}
 -2(uu_x,I^2u)&=(\partial_xI^2u,u^2)\nonumber\\
 &=\int_\R i\xi m^2(\xi)\widehat{u}(\xi) \widehat{u^2}(-\xi)\d \xi\nonumber\\
 &=\int_{\xi_1+\xi_2+\xi_3=0} i\xi_1 m^2(\xi_1)\widehat{u}(\xi_1) \widehat{u}(\xi_2)\widehat{u}(\xi_3),
\end{align}
where the integral can be understood as follows
$$
\int_{\xi_1+\xi_2+\xi_3=0}=\int_{\xi_1+\xi_2+\xi_3=0}\d\xi_1\d \xi_2=\int_{\xi_1+\xi_2+\xi_3=0}\d\xi_1\d \xi_3=\int_{\xi_1+\xi_2+\xi_3=0}\d\xi_2\d \xi_3.
$$
By the symmetry property of $\xi_1,\xi_2,\xi_3$, we can replace $\xi_1 m^2(\xi_1)$ by $\xi_2 m^2(\xi_2)$ or $\xi_3 m^2(\xi_3)$ in the integral \eqref{equ-law-6}, so we have
\begin{align}\label{equ-law-7}
 -2(uu_x,I^2u) = \frac{i}{3}\int_{\xi_1+\xi_2+\xi_3=0} \Big(\xi_1 m^2(\xi_1)+\xi_2 m^2(\xi_2)+\xi_3 m^2(\xi_3)\Big)\widehat{u}(\xi_1) \widehat{u}(\xi_2)\widehat{u}(\xi_3).
\end{align}
In particular, if $m(\xi)=1$, then $I$ is the identity operator, and \eqref{equ-law-7} shows that
\begin{align}\label{equ-law-8}
 2(uu_x,u) = \frac{i}{3}\int_{\xi_1+\xi_2+\xi_3=0} \Big(\xi_1 + \xi_2  +\xi_3  \Big)\widehat{u}(\xi_1) \widehat{u}(\xi_2)\widehat{u}(\xi_3)=0,
\end{align}
which gives a Fourier analysis proof of the fact $\int_\R u^2u_x\d x=0$.

Plugging \eqref{equ-law-4} into \eqref{equ-law-7} gives
\begin{align}\label{equ-law-9}
- 2(uu_x,I^2u) &= \frac{i}{3}\int_{\xi_1+\xi_2+\xi_3=0}\sum_{j=1}^3\xi_j(1+\frac{1}{2}\sum_{k=1}^\infty \frac{(2\sigma \xi_j)^{2k}}{(2k)!})  \widehat{u}(\xi_1) \widehat{u}(\xi_2)\widehat{u}(\xi_3)\nonumber\\
&=\frac{i}{6}\int_{\xi_1+\xi_2+\xi_3=0} \sum_{k=1}^\infty \frac{(2\sigma)^{2k}}{(2k)!} (\xi_1^{2k+1}+\xi_2^{2k+1}+\xi_3^{2k+1}) \widehat{u}(\xi_1) \widehat{u}(\xi_2)\widehat{u}(\xi_3).
\end{align}
Here in the last step we have used  \eqref{equ-law-8}.

To proceed, for every $k\geq 1$, we need to analyse the polynomial $\xi_1^{2k+1}+\xi_2^{2k+1}+\xi_3^{2k+1}$ on the hyperplane $\xi_1+\xi_2+\xi_3=0$.
\begin{lemma}\label{lem-factor}
If $\xi_1+\xi_2+\xi_3=0$, then
$$
\xi_1^{2k+1}+\xi_2^{2k+1}+\xi_3^{2k+1}=\xi_1\xi_2\xi_3\sum_{i+j=2k-2} \Big( \xi_1^i(-\xi_2)^j+\xi_1^i(-\xi_3)^j+\xi_2^i(-\xi_3)^j \Big).
$$
\end{lemma}
\begin{remark}
In the case $k=1$, Lemma \ref{lem-factor} recovers the well known identity
$$
\xi^3_1+\xi_2^3+\xi_3^3=3\xi_1\xi_2\xi_3, \quad \mbox{ if } \xi_1+\xi_2+\xi_3=0,
$$
which has been used in the well posedness of the KdV equation, see e.g. \cite{Tao}. In the case $k=2$, it is interesting to note that Lemma \ref{lem-factor} gives
$$
\xi^5_1+\xi_2^5+\xi_3^5=-5\xi_1\xi_2\xi_3(\xi_1\xi_2+\xi_1\xi_3+\xi_2\xi_3), \quad \mbox{ if } \xi_1+\xi_2+\xi_3=0.
$$
\end{remark}

\begin{proof}[Proof of Lemma \ref{lem-factor}]
We divide the proof into three steps.

{\bf Step 1. Factor $\xi_1$.} Since $\xi_1+\xi_2+\xi_3=0$, we have $\xi_2+\xi_3=-\xi_1$, thus
$$
\xi_2^{2k+1}+\xi_3^{2k+1}=(\xi_2+\xi_3)\sum_{i+j=2k} \xi_2^i(-\xi_3)^j=-\xi_1\sum_{i+j=2k} \xi_2^i(-\xi_3)^j.
$$
It follows that
\begin{align}\label{equ-fa-1}
\xi_1^{2k+1}+\xi_2^{2k+1}+\xi_3^{2k+1}&=\xi_1^{2k+1}-\xi_1\sum_{i+j=2k} \xi_2^i(-\xi_3)^j\nonumber\\
&=\xi_1\Big(\xi_1^{2k}-\sum_{i+j=2k} \xi_2^i(-\xi_3)^j\Big).
\end{align}

{\bf Step 2. Factor $\xi_2$.} Next we separate a factor $\xi_2$ out of $\xi_1^{2k}-\sum_{i+j=2k} \xi_2^i(-\xi_3)^j$. To this end, we split the sum $\sum_{i+j=2k} \xi_2^i(-\xi_3)^j$ as two parts, one has a factor $\xi_2$, and the other has no factor $\xi_2$, namely
$$
\sum_{i+j=2k} \xi_2^i(-\xi_3)^j=\xi_3^{2k}+\sum_{i+j=2k, i>0} \xi_2^i(-\xi_3)^j.
$$
It follows that
\begin{align}\label{equ-fa-2}
\xi_1^{2k}-\sum_{i+j=2k} \xi_2^i(-\xi_3)^j = \xi_1^{2k}-\xi_3^{2k}-\sum_{i+j=2k, i>0} \xi_2^i(-\xi_3)^j.
\end{align}
On the one hand, since $\xi_1+\xi_3=-\xi_2$, we find
\begin{align}\label{equ-fa-3}
 \xi_1^{2k}-\xi_3^{2k}&=(\xi_1+\xi_3)(\xi_1-\xi_3)\sum_{i+j=k-1}\xi_1^{2i}\xi_3^{2j}\nonumber\\
 &=-\xi_2(\xi_1-\xi_3)\sum_{i+j=k-1}\xi_1^{2i}\xi_3^{2j}.
\end{align}
On the other hand,  we have
\begin{align}\label{equ-fa-4}
-\sum_{i+j=2k, i>0} \xi_2^i(-\xi_3)^j&=-\Big(\xi_2^{2k}-\xi_2^{2k-1}\xi_3+\cdots+\xi_2^2\xi_3^{2k-2}-\xi_2\xi_3^{2k-1}\Big)\nonumber\\
&=-\xi_2\Big(\xi_2^{2k-1}-\xi_2^{2k-2}\xi_3+\cdots+\xi_2 \xi_3^{2k-2}- \xi_3^{2k-1}\Big)\nonumber\\
&=-\xi_2\sum_{i+j=2k-1}\xi_2^i(-\xi_3)^j.
\end{align}
Plugging \eqref{equ-fa-2}-\eqref{equ-fa-4} into \eqref{equ-fa-1} gives
\begin{align}\label{equ-fa-5}
\xi_1^{2k+1}+\xi_2^{2k+1}+\xi_3^{2k+1}=-\xi_1\xi_2\Big((\xi_1-\xi_3)\sum_{i+j=k-1}\xi_1^{2i}\xi_3^{2j}+\sum_{i+j=2k-1}\xi_2^i(-\xi_3)^j\Big).
\end{align}

{\bf Step 3. Factor $\xi_3$.} Denote
\begin{align}\label{equ-fa-6}
 \Theta = (\xi_1-\xi_3)\sum_{i+j=k-1}\xi_1^{2i}\xi_3^{2j}+\sum_{i+j=2k-1}\xi_2^i(-\xi_3)^j.
\end{align}
In a similar spirit as step 2, but focus on another factor $\xi_3$, we split $\Theta$ as
\begin{align}\label{equ-fa-7}
 \Theta & = -\xi_3 \sum_{i+j=k-1}\xi_1^{2i}\xi_3^{2j} + \xi_1 \sum_{i+j=k-1,j>0}\xi_1^{2i}\xi_3^{2j}+\xi_1^{2k-1}\nonumber\\
 & \quad +\sum_{i+j=2k-1,j>0}\xi_2^i(-\xi_3)^j +\xi_2^{2k-1}\nonumber\\
 & = -\xi_3 \sum_{i+j=k-1}\xi_1^{2i}\xi_3^{2j}+\xi_1^{2k-1}+\xi_2^{2k-1}\nonumber\\
 & \quad + \xi_1 \sum_{i+j=k-1,j>0}\xi_1^{2i}\xi_3^{2j}+\sum_{i+j=2k-1,j>0}\xi_2^i(-\xi_3)^j.
\end{align}
Now we analyse the terms on the right hand side of \eqref{equ-fa-7}. First, since $\xi_1+\xi_2=-\xi_3$, we have
\begin{align}\label{equ-fa-8}
 \xi_1^{2k-1}+\xi_2^{2k-1}=(\xi_1+\xi_2)\sum_{i+j=2k-2}\xi_1^i(-\xi_2)^j=-\xi_3\sum_{i+j=2k-2}\xi_1^i(-\xi_2)^j.
\end{align}
Also,
\begin{align}\label{equ-fa-9}
\sum_{i+j=2k-1,j>0}\xi_2^i(-\xi_3)^j=-\xi_3\sum_{i+j=2k-1,j>0}\xi_2^i(-\xi_3)^{j-1}=-\xi_3\sum_{i+j=2k-2}\xi_2^i(-\xi_3)^{j}.
\end{align}
Finally, a direct computation gives that
\begin{align}\label{equ-fa-10}
& -\xi_3 \sum_{i+j=k-1}\xi_1^{2i}\xi_3^{2j}  + \xi_1 \sum_{i+j=k-1,j>0}\xi_1^{2i}\xi_3^{2j} \nonumber\\
&=-\xi_3\Big( \sum_{i+j=k-1}\xi_1^{2i}\xi_3^{2j} - \sum_{i+j=k-1,j>0}\xi_1^{2i+1}\xi_3^{2j-1} \Big)\nonumber\\
&=-\xi_3\Big(\big(\xi_1^{2k-2}+\xi_1^{2k-4}\xi_3^2+\cdots+\xi_3^{2k-2}\big)- \big(\xi_1^{2k-3}\xi_3 +\xi_1^{2k-5}\xi_3^3+\cdots+\xi_1\xi_3^{2k-3}  \big) \Big)\nonumber\\
&=-\xi_3\Big( \xi_1^{2k-2}-  \xi_1^{2k-3}\xi_3+\xi_1^{2k-4}\xi_3^2-\xi_1^{2k-5}\xi_3^3+\cdots-\xi_1\xi_3^{2k-3}+\xi_3^{2k-2}  \Big)\nonumber\\
&= -\xi_3\sum_{i+j=2k-2}\xi_1^i(-\xi_3)^j.
\end{align}
It follows from \eqref{equ-fa-7}-\eqref{equ-fa-10} that
\begin{align}\label{equ-fa-11}
 \Theta =-\xi_3\sum_{i+j=2k-2}\Big(\xi_1^i(-\xi_2)^j+\xi_1^i(-\xi_3)^j+\xi_2^i(-\xi_3)^j \Big).
\end{align}
Finally, inserting \eqref{equ-fa-11} and \eqref{equ-fa-6} into \eqref{equ-fa-6} gives that
$$
\xi_1^{2k+1}+\xi_2^{2k+1}+\xi_3^{2k+1}=\xi_1\xi_2\xi_3\sum_{i+j=2k-2}\Big(\xi_1^i(-\xi_2)^j+\xi_1^i(-\xi_3)^j+\xi_2^i(-\xi_3)^j \Big),
$$
which completes the proof.
\end{proof}

\begin{lemma}\label{lem-fa-b}
Let $\sigma>0$. If $\xi_1+\xi_2+\xi_3=0$, then
$$
\sum_{k=1}^\infty \frac{(2\sigma)^{2k}}{(2k)!} (\xi_1^{2k+1}+\xi_2^{2k+1}+\xi_3^{2k+1})\lesssim \sigma^{\frac{3}{2}}|\xi_1\xi_2\xi_3|^{\frac{5}{6}}e^{\sigma(|\xi_1|+|\xi_2|+|\xi_3|)}.
$$
\end{lemma}
\begin{proof}
Thanks to Lemma \ref{lem-factor}, we have
\begin{multline}\label{equ-fa-b-1}
\sum_{k=1}^\infty \frac{(2\sigma)^{2k}}{(2k)!} (\xi_1^{2k+1}+\xi_2^{2k+1}+\xi_3^{2k+1})\\
=\sum_{k=1}^\infty \frac{(2\sigma)^{2k}}{(2k)!} \xi_1\xi_2\xi_3\sum_{i+j=2k-2}\Big(\xi_1^i(-\xi_2)^j+\xi_1^i(-\xi_3)^j+\xi_2^i(-\xi_3)^j \Big).
\end{multline}
Using the elementary inequality
$$
a^ib^j\leq a^{2k-2}+b^{2k-2}, \quad \forall a,b\geq 0, i+j=2k-2,
$$
we infer that
\begin{align}\label{equ-fa-b-2}
\left|\sum_{i+j=2k-2}\Big(\xi_1^i(-\xi_2)^j+\xi_1^i(-\xi_3)^j+\xi_2^i(-\xi_3)^j \Big)\right|\leq 2(2k-1)(\xi_1^{2k-2}+\xi_2^{2k-2}+\xi_3^{2k-2}).
\end{align}
Plugging \eqref{equ-fa-b-2} into \eqref{equ-fa-b-1} yields that
\begin{align}\label{equ-fa-b-3}
&\left|\sum_{k=1}^\infty \frac{(2\sigma)^{2k}}{(2k)!} (\xi_1^{2k+1}+\xi_2^{2k+1}+\xi_3^{2k+1})\right|\nonumber\\
&\leq 2\sum_{k=1}^\infty \frac{(2\sigma)^{2k}}{(2k)!} |\xi_1\xi_2\xi_3|(2k-1)(\xi_1^{2k-2}+\xi_2^{2k-2}+\xi_3^{2k-2})\nonumber\\
&\leq 2\sum_{k=1}^\infty \frac{(2\sigma)^{2k}}{(2k-1)!} |\xi_1\xi_2\xi_3| (\xi_1^{2k-2}+\xi_2^{2k-2}+\xi_3^{2k-2})\nonumber\\
&= 2\sum_{k=0}^\infty \frac{(2\sigma)^{2k+2}}{(2k+1)!} |\xi_1\xi_2\xi_3| (\xi_1^{2k}+\xi_2^{2k}+\xi_3^{2k})\nonumber\\
&=8\sigma^{\frac{3}{2}}|\xi_1\xi_2\xi_3|^{\frac{5}{6}}\Psi(\sigma \xi_1,\sigma \xi_2, \sigma \xi_3),
\end{align}
where
\begin{align}\label{equ-fa-b-4}
\Psi(\xi_1,\xi_2,\xi_3)= \sum_{k=0}^\infty \frac{2^{2k}}{(2k+1)!} |\xi_1\xi_2\xi_3|^{\frac{1}{6}} (\xi_1^{2k}+\xi_2^{2k}+\xi_3^{2k}).
\end{align}

If we can show that
\begin{align}\label{equ-fa-b-5}
\Psi(\xi_1,\xi_2,\xi_3)\lesssim e^{|\xi_1|+|\xi_2|+|\xi_3|},
\end{align}
then by \eqref{equ-fa-b-3}-\eqref{equ-fa-b-4},
$$
\left|\sum_{k=1}^\infty \frac{\sigma^{2k}}{(2k)!} (\xi_1^{2k+1}+\xi_2^{2k+1}+\xi_3^{2k+1})\right|\lesssim \sigma^{\frac{3}{2}}|\xi_1\xi_2\xi_3|^{\frac{5}{6}}e^{\sigma(|\xi_1|+|\xi_2|+|\xi_3|)},
$$
which completes the proof.

So it remains to prove \eqref{equ-fa-b-5}. To this end, we claim that
\begin{align}\label{equ-fa-b-6}
 \sum_{k=0}^\infty \frac{2^{2k}}{(2k+1)!} |\xi_1\xi_2\xi_3|^{\frac{1}{6}}  \xi_1^{2k}\lesssim e^{|\xi_1|+|\xi_2|+|\xi_3|}  \quad \mbox{ for } \xi_1+\xi_2+\xi_3=0.
\end{align}
In fact, if $\xi_1+\xi_2+\xi_3=0$, then
$$
\xi_1^{2k}=|\xi_1|^k|\xi_2+\xi_3|^k\leq |\xi_1|^k(|\xi_2|+|\xi_3|)^k
$$
and
$$
|\xi_1\xi_2\xi_3|^{\frac{1}{6}}=|(\xi_2+\xi_3)\xi_2\xi_3|^{\frac{1}{6}}\leq (|\xi_2|+|\xi_3|)^{\frac{1}{2}}.
$$
Thanks to these inequalities, we have
\begin{align}\label{equ-7-27}
 \sum_{k=0}^\infty \frac{2^{2k}}{(2k+1)!} |\xi_1\xi_2\xi_3|^{\frac{1}{6}}  \xi_1^{2k}&\leq  \sum_{k=0}^\infty \frac{2^{2k}}{(2k+1)!} |\xi_1|^{k} (|\xi_2|+|\xi_3|)^{k+\frac{1}{2}}.
\end{align}
Applying the fact $\max_{x>0}x^{a}e^{-x}=a^ae^{-a}$ with $a=k+\frac{1}{2}$ and $x=|\xi_2|+|\xi_3|$ shows
\begin{align}\label{equ-7-28}
 (|\xi_2|+|\xi_3|)^{k+\frac{1}{2}}  &\leq  \left( \frac{k+\frac{1}{2}}{e} \right)^{k+\frac{1}{2}}e^{|\xi_2|+|\xi_3|}\leq \left( \frac{k+1}{e} \right)^{k+\frac{1}{2}}e^{|\xi_2|+|\xi_3|} \nonumber\\
 & \lesssim \frac{1}{\sqrt{k}} \left( \frac{k+1}{e} \right)^{k+1}e^{|\xi_2|+|\xi_3|} \lesssim \frac{1}{k} \sqrt{k+1}\left( \frac{k+1}{e} \right)^{k+1}e^{|\xi_2|+|\xi_3|}\nonumber\\
 &\lesssim \frac{1}{k}(k+1)!e^{|\xi_2|+|\xi_3|}\lesssim k!e^{|\xi_2|+|\xi_3|},
\end{align}
where we used the Stirling formuala
\begin{align}\label{equ-7-28.5}
k!\sim \sqrt{ k}\left( \frac{k}{e} \right)^k, \quad \forall k \geq 0.
\end{align}
Inserting \eqref{equ-7-28} into \eqref{equ-7-27} and using \eqref{equ-7-28.5} again, we get
\begin{align*}
&\sum_{k=0}^\infty \frac{2^{2k}}{(2k+1)!} |\xi_1\xi_2\xi_3|^{\frac{1}{6}}  \xi_1^{2k}\leq  \sum_{k=0}^\infty \frac{2^{2k}(k!)^2}{(2k+1)!} \frac{|\xi_1|^{k}}{k!} e^{|\xi_2|+|\xi_3|}\\
&\leq \sup_{k\geq 0} \frac{2^{2k}(k!)^2}{(2k+1)(2k)!} \sum_{k\geq 0}\frac{|\xi_1|^{k}}{k!} e^{|\xi_2|+|\xi_3|}\\
&\lesssim  e^{|\xi_1|+|\xi_2|+|\xi_3|}\sup_{k\geq 0} \frac{1}{2k+1}  \frac{2^{2k}k(\frac{k}{e})^{2k}}{\sqrt{2k}(\frac{2k}{e})^{2k}}\\
&=  e^{|\xi_1|+|\xi_2|+|\xi_3|}\sup_{k\geq 0} \frac{k}{\sqrt{2k}(2k+1)}  \lesssim e^{|\xi_1|+|\xi_2|+|\xi_3|},
\end{align*}
which proves \eqref{equ-fa-b-6}.

 Similarly,
\begin{align}\label{equ-fa-b-7}
 \sum_{k=0}^\infty \frac{1}{(2k+1)!} |\xi_1\xi_2\xi_3|^{\frac{1}{6}}  \xi_2^{2k}+ \sum_{k=0}^\infty \frac{1}{(2k+1)!} |\xi_1\xi_2\xi_3|^{\frac{1}{6}}  \xi_3^{2k}\lesssim e^{|\xi_1|+|\xi_2|+|\xi_3|}.
\end{align}
Combining \eqref{equ-fa-b-6}-\eqref{equ-fa-b-7} gives \eqref{equ-fa-b-5}.
\end{proof}

Now we can state the main result in this section.

\begin{proposition}[Almost conservation law]\label{prop-law}
Let $\sigma>0$ and $I$ be the the Fourier multiplier given by \eqref{equ-law-1}.  Let $Iu_0\in H^1(\R)$ and let $Iu\in C([0,\delta];H^1)$ be the local solution obtained in Propostion \ref{prop-loc}. Then the solution of the BBM equation satisfies that
$$
 \sup_{t\in[0,\delta]}\|Iu(t)\|^2_{H^1}\leq \|Iu_0\|^2_{H^1}+C\delta \sigma^{\frac{3}{2}}\|Iu_0\|^3_{H^1}.
$$
\end{proposition}
\begin{proof}
Then combining \eqref{equ-law-5} and \eqref{equ-law-9} gives that
\begin{multline}\label{equ-law-20}
 \frac{\d}{\d t}\int_\R \Big( (Iu)^2+(\partial_xIu)^2 \Big)\d x\\
 = -\frac{i}{3}\int_{\xi_1+\xi_2+\xi_3=0} \sum_{k=1}^\infty \frac{\sigma^{2k}}{(2k)!} (\xi_1^{2k+1}+\xi_2^{2k+1}+\xi_3^{2k+1}) \widehat{u}(\xi_1) \widehat{u}(\xi_2)\widehat{u}(\xi_3).
\end{multline}
Thanks to Lemma \ref{lem-fa-b}, the right hand side of \eqref{equ-law-20} can be bounded as
\begin{align}\label{equ-law-21}
 \mbox{ RHS of } \eqref{equ-law-20}
 \lesssim \sigma^{\frac{3}{2}}\int_{\xi_1+\xi_2+\xi_3=0}|\xi_1\xi_2\xi_3|^{\frac{5}{6}}e^{\sigma(|\xi_1|+|\xi_2|+|\xi_3|)}
 |\widehat{u}(\xi_1)\widehat{u}(\xi_2)\widehat{u}(\xi_3)|.
\end{align}
Denote by
\begin{align}\label{equ-law-22}
v=\mathcal {F}^{-1} |\widehat{u}(\xi)|
\end{align}
and
\begin{align}\label{equ-law-23}
D^\beta v=\mathcal {F}^{-1} (|\xi|^\beta\widehat{v}(\xi)), \quad e^{\sigma|D|}v=\mathcal {F}^{-1} (e^{\sigma|\xi|}\widehat{v}(\xi)).
\end{align}
Then by Parsaval formula, we rewrite \eqref{equ-law-21} as
\begin{align}\label{equ-law-24}
 \eqref{equ-law-21}
 = \sigma^{\frac{3}{2}}\int_{\R}\Big(D^{\frac{5}{6}}e^{\sigma|D|}v\Big)^3\d x\leq \sigma^{\frac{3}{2}} \|D^{\frac{5}{6}}e^{\sigma|D|}v\|^3_{L^3(\R)}.
\end{align}
By Sobolev embedding theorem, we have
\begin{align}\label{equ-law-25}
\|D^{\frac{5}{6}}e^{\sigma|D|}v\|_{L^3(\R)}\lesssim \|De^{\sigma|D|}v\|_{L^2(\R)}=\||\xi|e^{\sigma|\xi|}|\widehat{u}|\|_{L^2(\R)}\leq \|Iu\|_{H^1},
\end{align}
where we used \eqref{equ-law-22},\eqref{equ-law-23} and the definition of the $I$ operator.

Plugging \eqref{equ-law-24},\eqref{equ-law-25} and \eqref{equ-law-21} into \eqref{equ-law-20} yields that for some numerical constant $C_1>0$
\begin{align}\label{equ-law-26}
 \frac{\d}{\d t}\|Iu\|^2_{H^1} \leq C_1\sigma^{\frac{3}{2}}\|Iu\|^3_{H^1}, \quad t\in [0,\delta].
\end{align}
Integrating \eqref{equ-law-26} with respect to time we obtain
\begin{align}\label{equ-law-27}
 \sup_{t\in[0,\delta]}\|Iu(t)\|^2_{H^1} \leq \|Iu_0\|^2_{H^1}+\int_0^\delta C_1\sigma^{\frac{3}{2}}\|Iu(\tau)\|^3_{H^1}\d \tau.
\end{align}
Thanks to the local well posedness, we have
$$
\sup_{t\in[0,\delta]}\|Iu(\tau)\|_{H^1}\leq C_2\|Iu_0\|_{H^1}.
$$
Then we deduce from \eqref{equ-law-27} that
\begin{align}\label{equ-law-28}
  \sup_{t\in[0,\delta]}\|Iu(t)\|^2_{H^1} \leq \|Iu_0\|^2_{H^1}+ C_1C_2^3\delta\sigma^{\frac{3}{2}}\|Iu_0\|^3_{H^1}.
\end{align}
This completes the proof.
\end{proof}

\begin{remark}\label{rem-Tes}
We give a sketch here that what happens if we follow the argument of Dufera, Mebrate and Tesfahun in \cite{Tes22}. To this end, define another $I$-operator
$$
Iu = \mathcal {F}^{-1}(\cosh(\sigma|\xi|)\widehat{u}(\xi)).
$$
Multiplying \eqref{bbm} with $2I^2u$ and integrating over $\R$ gives that
\begin{align}\label{equ-rem-1}
 \frac{\d}{\d t}\|Iu\|^2_{H^1}=-2(uu_x,I^2u)=-2(uu_x,(I^2-1)u)=(u^2,\partial_x(I^2-1)u),
\end{align}
where we used the fact $\int_\R u^2u_x\d x=0$. By Parseval idenity we have
\begin{align}\label{equ-rem-2}
 (u^2,\partial_x(I^2-1)u)&=\int_{\xi_1+\xi_2+\xi_3=0}i\xi_1(\cosh^2(\sigma|\xi_1|)-1)\widehat{u}(\xi_1)
\widehat{u}(\xi_2)\widehat{u}(\xi_3)\nonumber\\
&=\int_{\xi_1+\xi_2+\xi_3=0}\frac{i}{2}\xi_1(\cosh(2\sigma|\xi_1|)-1)\widehat{u}(\xi_1)
\widehat{u}(\xi_2)\widehat{u}(\xi_3).
\end{align}
Thanks to the inequality
$$
\cosh (\sigma |\xi|)-1\leq (\sigma|\xi|)^{2\alpha} \cosh(\sigma|\xi|), \quad 0\leq \alpha\leq 1,
$$
we deduce from \eqref{equ-rem-1}-\eqref{equ-rem-2} that
\begin{align*}
& \frac{\d}{\d t}\|Iu\|^2_{H^1}\lesssim \int_{\xi_1+\xi_2+\xi_3=0}|\xi_1|(\sigma|\xi_1|)^{2\alpha} \cosh(2\sigma|\xi_1|)|\widehat{u}(\xi_1)
\widehat{u}(\xi_2)\widehat{u}(\xi_3)|\\
&\lesssim \sigma^{2\alpha}\int_{\xi_1+\xi_2+\xi_3=0}|\xi_1|(|\xi_2|^{2\alpha}+|\xi_3|^{2\alpha}) \cosh(\sigma|\xi_1|)\cosh(\sigma|\xi_2|)\cosh(\sigma|\xi_3|)|\widehat{u}(\xi_1)
\widehat{u}(\xi_2)\widehat{u}(\xi_3)|\\
&\lesssim \sigma^{2\alpha} \|DIu\|_{L^2(\R)}\|D^{2\alpha}Iu\|_{L^2(\R)}\|Iu\|_{L^\infty(\R)}\\
&\lesssim \sigma^{2\alpha} \|Iu\|^3_{H^1(\R)}
\end{align*}
provided that $2\alpha\leq 1$. The best conservation law in this way is then
$$
\|Iu(\delta)\|^2_{H^1}\leq \|Iu_0\|^2_{H^1}+C\sigma^{2\alpha}\|Iu_0\|^2_{H^1}.
$$
This implies the analytic radius lower bound $\sigma(t)\geq ct^{-1}$, which is worse than that in  Theorem \ref{thm-1}.
\end{remark}

\section{Generations to fractional BBM equation}

In this section, we apply our argument to more general BBM equations. Let $\alpha$ be a real number such that
\begin{align}\label{equ-frac-1}
\alpha>1.
\end{align}
Consider the fractional BBM equation on the real line
\begin{align}\label{fbbm}
u_t+D^\alpha u_t +  u_x + uu_x=0, \quad u(0,x)=u_0(x).
\end{align}
Here $u$ is a real valued function on $(t,x)\in \R^2$, the fractional differential operator $D^\alpha$ is defined by
$$
D^\alpha u = \mathcal {F}^{-1} (|\xi|^\alpha \widehat{u}(\xi)).
$$
In the case $1<\alpha\leq 2$, it is shown that \eqref{fbbm} is globally well posed \cite{bona-chen09,wang18} in $H^s(\R)$ with $s\geq \max\{0,\frac{3}{2}-\alpha\}$ and ill posed \cite{panthee} in $H^s(\R)$ for  $s< \max\{0,\frac{3}{2}-\alpha\}$. This implies, in particular, that \eqref{fbbm} is globally well posed in $H^{s}(\R)$ with $s\geq \frac{\alpha}{2}$, here we used the fact \eqref{equ-frac-1}. In the sequel, we shall focus on the global well posedness of \eqref{fbbm} in Gevrey spaces $G^{\sigma,s}$ with $s=\frac{\alpha}{2}$. Here and below, for every $\sigma\geq 0$ and $s\in \R$,  $G^{\sigma,s}$  denotes the Banach space endowed with the norm
\begin{align}\label{equ-ana-norm-1}
\|u\|_{G^{\sigma,s}}=\|(1+|\xi|)^se^{\sigma|\xi|}\widehat{u}\|_{L^2(\R)}.
\end{align}
We do not pursue the goal to lower the value of $s$, since the size of $\sigma$ plays a dominant role.
The main result in this section reads as follows.

\begin{theorem}\label{thm-f}
Assume that \eqref{equ-frac-1} holds and $u_0\in G^{\sigma_0,\frac{\alpha}{2}}$ for some $\sigma_0>0$. Then the fractional BBM equation \eqref{fbbm} has a global solution $u\in C([0,\infty);G^{\sigma(t),\frac{\alpha}{2}})$ with
$$
\sigma(t)\geq ct^{-\mu}, \quad t\to \infty,
$$
where $c>0$ depends only on $\sigma_0$ and $\|u_0\|_{G^{\sigma_0,\frac{\alpha}{2}}}$, and
$$
\mu= \left\{
\begin{array}{ll}
\frac{2}{3(\alpha-1)}&\quad \mbox{ if } 1<\alpha<\frac{7}{3},\\
\frac{1}{2}&\quad \mbox{ if } \alpha\geq \frac{7}{3}.
\end{array}
\right.
$$
\end{theorem}

\begin{remark}
It is proved by Bona and Gruji\'{c} in \cite{BG} that the analytic radius $\sigma(t)$ of solutions to \eqref{fbbm} has the lower bound
\begin{align}\label{equ-f-rem-1}
\sigma(t)\geq ct^{-\frac{1}{\alpha-1}}, \quad t\to \infty.
\end{align}
The approach used in \cite{BG} is different from that in this paper. Clearly, in the case $1<\alpha<\frac{7}{3}$, Theorem \ref{thm-f} is an improvement of \eqref{equ-f-rem-1}. However, if $\alpha$ is large enough, the lower bound in Theorem \ref{thm-f} is worse. Thus, it remains open that whether the approach used in this paper leads to a better bound than \eqref{equ-f-rem-1}.
\end{remark}

The rest of this section is devoted to proving Theorem \ref{thm-f}. Since the main idea is similar to Theorem \ref{thm-1}, we only give a sketch here.

As in Section 2, we let
\begin{align}\label{equ-frac-2}
Iu=\mathcal {F}^{-1}(m(\xi)\widehat{u}(\xi)), \quad m(\xi)=\frac{1}{2}(e^{\sigma \xi}+e^{-\sigma\xi})
\end{align}
and
\begin{align}\label{equ-frac-3}
\varphi(D)=\partial_x(1+D^\alpha)^{-1}.
\end{align}
\begin{lemma}\label{lem-bi-f}
For all $\sigma\geq 0$, we have
$$
\|\varphi(D) I (uv)\|_{H^{\frac{\alpha}{2}}(\R)}\lesssim \|Iu\|_{H^{\frac{\alpha}{2}}(\R)}\|Iv\|_{H^{\frac{\alpha}{2}}(\R)}.
$$
\end{lemma}
\begin{proof}
Since $\alpha>1$, by Plancherel theorem, we have
$$
\|\varphi(D) I (uv)\|_{H^{\frac{\alpha}{2}}(\R)}\lesssim \|I (uv)\|_{H^{\frac{\alpha}{2}}(\R)}.
$$
So it remains to show that
\begin{align}\label{equ-frac-4}
\|I (uv)\|_{H^{\frac{\alpha}{2}}(\R)}\lesssim \|Iu\|_{H^{\frac{\alpha}{2}}(\R)}\|Iv\|_{H^{\frac{\alpha}{2}}(\R)}.
\end{align}
But we know $H^{\frac{\alpha}{2}}(\R)$ is an algebra, namely
\begin{align}\label{equ-frac-5}
\|uv\|_{H^{\frac{\alpha}{2}}(\R)}\lesssim \|u\|_{H^{\frac{\alpha}{2}}(\R)}\|v\|_{H^{\frac{\alpha}{2}}(\R)}.
\end{align}
Then \eqref{equ-frac-4} follows from \eqref{equ-frac-5} and the fact
$$
m(\xi)\leq 2m(\xi_1)m(\xi_2), \quad \mbox{ if } \xi=\xi_1+\xi_2.
$$
This completes the proof.
\end{proof}

With Lemma \ref{lem-bi-f} in hand, we rewrite the fractional BBM equation \eqref{fbbm} as
$$
u(t)=e^{-t\varphi(D)}u_0-\frac{1}{2}\int_0^t e^{-(t-\tau)\varphi(D)}\varphi(D)u^2\d \tau,
$$
and then apply the contraction mapping principle to obtain a unique solution such that
$Iu\in C([0,\delta];H^{\frac{\alpha}{2}}(\R))$ satisfying
\begin{align}\label{equ-frac-6}
\sup_{t\in[0,\delta]}\|Iu(t)\|_{H^{\frac{\alpha}{2}}(\R)}\leq 2\|Iu_0\|_{H^{\frac{\alpha}{2}}(\R)}
\end{align}
with lifespan
\begin{align}\label{equ-frac-7}
\delta \sim \frac{1}{\|Iu_0\|_{H^{\frac{\alpha}{2}}(\R)}}.
\end{align}

Moreover, we have the following conservation law.
\begin{lemma}\label{lem-law-2}
Assume that $Iu_0\in H^{\frac{\alpha}{2}}(\R)$ for some $\sigma>0$. Then the solution of \eqref{fbbm} satisfies that
$$
E(\delta)\leq E(0)+C\delta \sigma^\beta \|Iu_0\|^3_{H^{\frac{\alpha}{2}}(\R)},
$$
where the energy $E(t)=\int_\R (|Iu(t)|^2+|D^{\frac{\alpha}{2}}Iu(t)|^2)\d x$ and
$$
\beta= \left\{
\begin{array}{ll}
\frac{3}{2}(\alpha-1)&\quad \mbox{ if } 1<\alpha<\frac{7}{3},\\
2&\quad \mbox{ if } \alpha\geq \frac{7}{3}.
\end{array}
\right.
$$
\end{lemma}
\begin{proof}
Similar to \eqref{equ-law-20} we have
\begin{align}\label{equ-frac-8}
 \frac{\d}{\d t}E(t)
 = -\frac{i}{3}\int_{\xi_1+\xi_2+\xi_3=0} \sum_{k=1}^\infty \frac{(2\sigma)^{2k}}{(2k)!} (\xi_1^{2k+1}+\xi_2^{2k+1}+\xi_3^{2k+1}) \widehat{u}(\xi_1) \widehat{u}(\xi_2)\widehat{u}(\xi_3).
\end{align}
If $\xi_1+\xi_2+\xi_3=0$, then by Lemma \ref{lem-factor},
\begin{align}\label{equ-frac-9}
&\left|\sum_{k=1}^\infty \frac{(2\sigma)^{2k}}{(2k)!} (\xi_1^{2k+1}+\xi_2^{2k+1}+\xi_3^{2k+1})\right|\nonumber\\
&\lesssim \left| \sum_{k=1}^\infty \frac{\sigma^{2k}}{(2k)!} \xi_1\xi_2\xi_3\sum_{i+j=2k-2}\Big(\xi_1^i(-\xi_2)^j+\xi_1^i(-\xi_3)^j+\xi_2^i(-\xi_3)^j\Big)\right|\nonumber\\
&\lesssim \sum_{k=0}^\infty \frac{(2\sigma)^{2k+2}}{(2k+2)!} |\xi_1\xi_2\xi_3|\sum_{i+j=2k}\left|\Big(\xi_1^i(-\xi_2)^j+\xi_1^i(-\xi_3)^j+\xi_2^i(-\xi_3)^j\Big)\right|\nonumber\\
&=4\sigma^{2-3(1-\varepsilon_0)}|\xi_1\xi_2\xi_3|^{\varepsilon_0}\Psi(\sigma \xi_1,\sigma \xi_2, \sigma \xi_3)
\end{align}
where
\begin{align}\label{equ-frac-10}
\varepsilon_0=\min\left\{\frac{\alpha}{2}-\frac{1}{6}, 1\right\},
\end{align}
and
\begin{align}\label{equ-frac-11}
\Psi(\xi_1, \xi_2, \xi_3)=\sum_{k=0}^\infty \frac{2^{2k}}{(2k+2)!} |\xi_1\xi_2\xi_3|^{1-\varepsilon_0}\sum_{i+j=2k}\left|\Big(\xi_1^i(-\xi_2)^j+\xi_1^i(-\xi_3)^j+\xi_2^i(-\xi_3)^j\Big)\right|.
\end{align}

Since $\alpha>1$, by definition \eqref{equ-frac-10}, we see $1-\varepsilon_0\in[0,2/3)$, thus
\begin{align}\label{equ-7-31-0}
|\xi_1\xi_2\xi_3|^{1-\varepsilon_0}\lesssim 1 + |\xi_1\xi_2\xi_3|^{\frac{2}{3}}.
\end{align}
Similar to \eqref{equ-fa-b-5}, one has
$$
\sum_{k=0}^\infty \frac{2^{2k}}{(2k+2)!} \sum_{i+j=2k}\left|\Big(\xi_1^i(-\xi_2)^j+\xi_1^i(-\xi_3)^j+\xi_2^i(-\xi_3)^j\Big)\right|\lesssim e^{|\xi_1|+|\xi_2|+|\xi_3|}.
$$
By definition \eqref{equ-frac-11}, we have the bound
\begin{align}\label{equ-7-31-1}
\Psi(\xi_1, \xi_2, \xi_3)\lesssim e^{|\xi_1|+|\xi_2|+|\xi_3|}.
\end{align}
provided we show that
\begin{align}\label{equ-7-31-2}
\sum_{k=0}^\infty \frac{2^{2k}}{(2k+2)!} |\xi_1\xi_2\xi_3|^{\frac{2}{3}} \sum_{i+j=2k}\left|\Big(\xi_1^i(-\xi_2)^j+\xi_1^i(-\xi_3)^j+\xi_2^i(-\xi_3)^j\Big)\right|\lesssim e^{|\xi_1|+|\xi_2|+|\xi_3|}.
\end{align}

To estimate \eqref{equ-7-31-2}, without loss of generality, we assume $|\xi_1|\geq |\xi_2|\geq |\xi_3|$. Note that $\xi_1+\xi_2+\xi_3=0$, 
$$
|\xi_1\xi_2\xi_3|^{\frac{2}{3}}\leq |(|\xi_2|+|\xi_3|)\xi_2\xi_3|^{\frac{2}{3}}\leq(|\xi_2|+|\xi_3|)^2,
$$
thus the left hand side of \eqref{equ-7-31-2} can be bounded by
\begin{align}\label{equ-7-31-3}
\mbox{ LHS of } \eqref{equ-7-31-2} &\lesssim \sum_{k=0}^\infty \frac{2^{2k}}{(2k+2)!}  (|\xi_2|+|\xi_3|)^2 \sum_{i+j=2k}|\xi_1|^i|\xi_2|^j\nonumber\\
&\lesssim \sum_{k=0}^\infty \frac{2^{2k}}{(2k+2)!}  (|\xi_2|+|\xi_3|)^2 \sum_{i+j=2k,i\geq k}|\xi_1|^i|\xi_2|^j,
\end{align}
where in the last step we used the fact $|\xi_1|\geq |\xi_2|$. For $i\geq k$ and $\xi_1+\xi_2+\xi_3=0$, we rewrite
\begin{align}\label{equ-7-31-4}
  &(|\xi_2|+|\xi_3|)^2\sum_{i+j=2k,i\geq k}|\xi_1|^i|\xi_2|^j\nonumber\\
  &=(|\xi_2|+|\xi_3|)^2|\xi_1|^k\sum_{i+j=2k,i\geq k}|\xi_1|^{i-k}|\xi_2|^j\nonumber\\
  &\leq (|\xi_2|+|\xi_3|)^2 |\xi_1|^k\sum_{i+j=k}|\xi_1|^{i}|\xi_2|^j\nonumber\\
  & \leq |\xi_1|^k\sum_{i+j=k}(|\xi_2|+|\xi_3|)^{i+2}|\xi_2|^j\nonumber\\
  &\leq |\xi_1|^k\sum_{i+j=k+2}(|\xi_2|+|\xi_3|)^{i}|\xi_2|^j\nonumber\\
  & \leq  |\xi_1|^k\sum_{i+j=k+2}\sum_{\ell_1+\ell_2=i}\frac{i!}{\ell_1!\ell_2!}|\xi_3|^{\ell_1}|\xi_2|^{j+\ell_2}.
\end{align}
Using the elementary inequality (see \cite[Lemma 3.4]{wangm19})
$$
n_2!n_3!\leq n_1!n_4!,\quad \forall 0\leq n_1\leq n_2\leq n_3\leq n_4, n_1+n_4=n_2+n_3
$$
we have
$$
\frac{i!}{\ell_1!\ell_2!}\leq \frac{(k+2)!}{\ell_1!(j+\ell_2)!}.
$$
This, together with \eqref{equ-7-31-4}, gives
\begin{align}\label{equ-7-31-5}
  (|\xi_2|+|\xi_3|)^2\sum_{i+j=2k,i\geq k}|\xi_1|^i|\xi_2|^j  & \leq  |\xi_1|^k\sum_{i+j=k+2}\sum_{\ell_1+\ell_2=i}\frac{(k+2)!}{\ell_1!(j+\ell_2)!}|\xi_3|^{\ell_1}|\xi_2|^{j+\ell_2}\nonumber\\
  &=(k+2)!|\xi_1|^k\sum_{i+j=k+2} \frac{|\xi_3|^i}{i!}\frac{|\xi_2|^j}{j!}.
\end{align}
Combining \eqref{equ-7-31-3} and \eqref{equ-7-31-5}, we obtain
\begin{align}\label{equ-7-31-6}
\mbox{ LHS of } \eqref{equ-7-31-2} &\lesssim \sum_{k=0}^\infty \frac{2^{2k}}{(2k+2)!} (k+2)!|\xi_1|^k\sum_{i+j=k+2} \frac{|\xi_3|^i}{i!}\frac{|\xi_2|^j}{j!}\nonumber\\
&\lesssim \sum_{k=0}^\infty \frac{2^{2k}}{(2k)!} k!|\xi_1|^k\sum_{i+j=k+2} \frac{|\xi_3|^i}{i!}\frac{|\xi_2|^j}{j!}.
\end{align}
By Stirling formuala \eqref{equ-7-28.5} we have
$$
\frac{2^{2k}}{(2k)!} k!\lesssim \frac{\sqrt{k}}{k!}
$$
and similar to \eqref{equ-7-28} we have
$$
\frac{\sqrt{k}}{k!}|\xi_1|^k\lesssim e^{|\xi_1|}
$$
then we deduce from \eqref{equ-7-31-6} that
\begin{align}\label{equ-7-31-7}
\mbox{ LHS of } \eqref{equ-7-31-2} &\lesssim e^{|\xi_1|}\sum_{k=0}^\infty  \sum_{i+j=k+2}  \frac{|\xi_3|^i}{i!}\frac{|\xi_2|^j}{j!}\leq e^{|\xi_1|}\sum_{k=2}^\infty  \sum_{i+j=k}  \frac{|\xi_3|^i}{i!}\frac{|\xi_2|^j}{j!}\nonumber\\ 
&\leq e^{|\xi_1|}\left(\sum_{i=0}^\infty\frac{|\xi_3|^i}{i!}\right) \left(\sum_{j=0}^\infty\frac{|\xi_2|^j}{j!}\right)=e^{|\xi_1|+|\xi_2|+|\xi_3|}. 
\end{align}

This proves \eqref{equ-7-31-2} and thus \eqref{equ-7-31-1} holds. Then it follows from \eqref{equ-frac-8}, \eqref{equ-frac-9} and \eqref{equ-frac-11} that
\begin{align}\label{equ-frac-12}
 \frac{\d}{\d t}E(t)
 &\lesssim \sigma^{2-3(1-\varepsilon_0)}\int_{\xi_1+\xi_2+\xi_3=0}|\xi_1\xi_2\xi_3|^{\varepsilon_0}e^{\sigma(|\xi_1|+|\xi_2|+|\xi_3|)} |\widehat{u}(\xi_1) \widehat{u}(\xi_2)\widehat{u}(\xi_3)|\nonumber\\
 &\lesssim \sigma^{2-3(1-\varepsilon_0)} \|D^{\varepsilon_0}Iv\|^3_{L^3(\R)}\lesssim \sigma^{2-3(1-\varepsilon_0)} \|Iu\|^3_{H^{\frac{\alpha}{2}}(\R)},
\end{align}
where $v=\mathcal {F}^{-1}|\widehat{u}|$, and we used the Sobolev embedding (since \eqref{equ-frac-10})
$$
\|D^{\varepsilon_0}u\|_{L^3(\R)}\lesssim \|u\|_{H^{\frac{\alpha}{2}}(\R)}.
$$
Integrating \eqref{equ-frac-12} over $[0,\delta]$, using the local bound \eqref{equ-frac-6}, noting
$$
2-3(1-\varepsilon_0)= \left\{
\begin{array}{ll}
\frac{3}{2}(\alpha-1)&\quad \mbox{ if } 1<\alpha<\frac{7}{3},\\
2&\quad \mbox{ if } \alpha\geq \frac{7}{3}.
\end{array}
\right.
$$
we conclude the desired bound.
\end{proof}

Thanks to Lemma \ref{lem-law-2} and the local well posedness, similar to the proof of Theorem \ref{thm-1}, we find that the solution of \eqref{fbbm} satisfies
$$
u\in C([0,\infty); G^{\sigma(t),\frac{\alpha}{2}})
$$
with
$$
\sigma(t)\geq ct^{-\frac{1}{\beta}}, \quad t\to \infty,
$$
where $\beta$ is the same as that in Lemma \ref{lem-law-2}. This proves Theorem \ref{thm-f}.

\section*{Acknowledgements}

This work is partially supported by the National Natural Science Foundation of China under grant No.12171442, and the Fundamental Research Funds for the Central Universities, China University of Geosciences(Wuhan) under grant No.CUGSX01.

\end{document}